\documentclass{amsart}
\usepackage{epsf,amscd,amsmath,amssymb,amsxtra,abbrViro}
\usepackage{hyperref}
\usepackage[xdvi]{graphicx}

\theoremstyle{plain}
\newtheorem{Th}{\itshape Theorem}[section]

\newtheorem{lem}[Th]{\itshape Lemma}

\newtheorem{ATh}{\itshape Theorem}[subsection]

\newtheorem{Acor}[ATh]{\itshape Corollary}

\newtheorem{Alem}[ATh]{\itshape Lemma}

\begin{document}
\title[Twisted acyclicity of a circle and 
link signatures]{Twisted acyclicity of a circle\\
and  link signatures}
\author{OLEG VIRO}
\address{\noindent
Mathematics Department, Stony Brook University,  Stony Brook
NY 11794-3651, USA\newline
\phantom{aa'}Mathematical Institute, 
Fontanka 27, St.~Petersburg, 191023,  Russia.}
                                
\begin{abstract}
Homology of the circle with non-trivial local coefficients 
is trivial. From this well-known fact we deduce geometric 
corollaries concerning links of codimension two. In particular,
the Murasugi-Tristram signatures are extended to invariants of
links formed of arbitrary oriented closed codimension two 
submanifolds of an odd-dimensional sphere. The novelty is 
that the submanifolds are not assumed to be disjoint, but
are transversal to each other, and the signatures are parametrized
by points of the whole torus. Murasugi-Tristram inequalities 
and their generalizations are also extended to this setup. 
\end{abstract}
\maketitle

\section{Introduction}\label{s1} %{s0}

%\subsection{The subject of the paper}\label{s1.1} %{s0.1}

The goal of this paper is to simplify and generalize a part 
of classical link theory based on various signatures of links
(defined by Trotter \cite{Trot} Murasugi \cite{Mura1},\cite{Mura2}, 
Tristram \cite{Trist}, Levine \cite{Levine1} \cite{Levine2}, 
Smolinsky \cite{Smolinsky}, Florens \cite{Florens1} and 
Cimasoni and Florens \cite{CimaFlor}). 
This part is known for its relations to topology of 4-dimensional
manifolds, see \cite{Trist}, \cite{Viro1}, \cite{Viro2} \cite{Gilmer},
\cite{KaufTayl} 
and applications in 
topology of real algebraic curves \cite{Orevkov1}, \cite{Orevkov2} and 
\cite{Florens1}. 

Similarity of the signatures to the new invariants
\cite{Rasm}, \cite{OzsSz1}, 
which were defined in the new frameworks of link homology theories 
and had spectacular applications \cite{Rasm}, \cite{Livingst}, 
\cite{Shum} to problems on classical link cobordisms, gives a new 
reason to revisit the old theory.    

There are two ways to introduce the signatures: 
the original 3-dimensional, via Seifert surface and Seifert form, and
4-dimensional, via the intersection form of the cyclic coverings of
4-ball branched over surfaces. I believe, this paper clearly
demonstrates advantages of the latter, 4-dimensional approach, which
provides more conceptual definitions, easily working in the situations
hardly available for the Seifert form approach. 

In the generalization considered here the classical links are 
replaced by collections  of transversal to each other oriented 
submanifolds of codimension two.

Technically the work is based on a systematic use of twisted homology
and the intersection 
forms in the twisted homology. Only the simplest kinds of 
twisted homology is used, the one with coefficients in $\C$, see 
Appendix.   

\subsection{Twisted acyclicity of a circle}\label{s1.2} %{s0.2}
A key property of twisted homology, which makes the whole story
possible, is the following well-known fact, which I call 
{\sfit twisted acyclicity of a circle\/}:

{\it Twisted homology of a circle with coefficients in $\C$ and 
non-trivial monodromy vanishes.}

This implies that the twisted homology of this kind completely
ignores parts of the space formed by circles along which
the monodromy of the coefficient system is non-trivial 
(for precise and detailed formulation see Section \ref{sT.2}).

\subsection{How the acyclicity works}\label{s1.3} %{s0.3}
In particular, twisted acyclicity of a circle implies 
that the complement of a tubular neighborhood of a
link looks like a closed manifold, because the boundary, being 
fibered to circles, is invisible for the twisted homology. 

Moreover, the same holds true for a collection of pairwise 
transversal generically immersed closed manifolds of codimension 2 in 
arbitrary closed manifold, provided the monodromy around 
each manifold is non-trivial.
The twisted homology does not feel the intersection of the
submanifolds as a singularity. 

The complement of a cobordism between such immersed 
links looks (again, from the point of view of twisted 
homology) like a compact cobordism between closed manifolds.

This, together with classical results about signatures of manifolds 
and relations between twisted homology and
homology with constant coefficients, allows us to deal with
a link of codimension two as if it was a single closed manifold.

\subsection{Organization of the paper}\label{s1.4}
I cannot assume the twisted homology well-known to the reader,
and review the material related to it. Of course, the material on 
non-twisted homology is not reviewed. The review is
limited to a very special twisted homology, the one with 
complex coefficients.
More general twisted homology is not needed here.

The review is postponed to appendices. 
The reader somehow familiar with twisted homology may visit
this section when needed.  The experts are invited
to look through appendices, too. 

We begin in Section \ref{s2} with a detailed exposition restricted 
to the classical links. Section \ref{s3} is devoted to higher
dimensional generalization, including motivation for our choice of 
the objects. Section \ref{s4} is devoted to {\sl span inequalities\/}, 
that is,
restrictions on homology of submanifolds of the ball, which span 
a given link contained in the boundary of the ball. Section \ref{s5} 
is devoted to {\sl slice inequalities\/}, which are restrictions on
homology of a link with given transversal intersection with a sphere
of codimension one.  

\section{In the classical dimension}\label{s2}

\subsection{Classical knots and links.}\label{s2.1} %{s0.4.1}
Recall that a {\sfit classical knot\/} is a smooth simple closed 
curve in the 3-sphere $S^3$. This is how one usually defines 
classical knots. However it is not the curve per se that 
is really considered in the classical knot theory, but
rather its placement in $S^3$. Classical knots incarnate the
idea of knottedness: both the curve and $S^3$ are topologically 
standard, but the position of the curve in $S^3$ may be
arbitrarily complicated topologically.
Therefore a classical knot 
is rather a pair $(S^3,K)$, where $K$ is a smooth
submanifold of $S^3$ diffeomorphic to $S^1$.            

A {\sfit classical link\/} is a pair $(S^3,L)$, 
where $L$ is a smooth closed one-dimensional submanifold 
of $S^3$. If $L$ is connected, then this is a knot.  

\subsection{Twisted homology of a classical link 
exterior}\label{s2.2}
An {\sfit exterior\/} of a classical link $(S^3,L)$ is the complement 
of an open tubular neighborhood of $L$. This is a compact 3-manifold
with boundary. The boundary is the boundary of the tubular neighborhood
of $L$. Hence, this is the total space of a locally trivial fibration 
over $L$ with fiber $S^1$. An exterior $X(L)$ is a deformation retract
of the complement $S^3\sminus L$. It's a nice replacement of 
$S^3\sminus L$, because $\Int X(L)$ is homeomorphic to 
$S^3\sminus L$, but $X(L)$ is compact manifold and has a nice boundary.

If $L$ consists of $m$ connected components, $L=K_1\cup\dots\cup K_m$,
then by the Alexander duality $H_0(X(L))=\Z$, $H_1(X(L))=\Z^m$, 
$H_2(X(L))=\Z^{m-1}$ 
and $H_i(X(L))=0$ for $i\ne 0,1,2$. The group $H_1(X(L))$ is dual to 
$H_1(L)$ with respect to the Alexander linking pairing 
$H_1(L)\times H_1(X(L))\to\Z$.  Hence a basis of $H_1(L)$ defines a dual
basis in $H_1(X(L))$. An orientation of $L$ determines a basis
$[K_1]$, \dots, $[K_m]$ of $H_1(L)$, and the dual basis of $H_1(X(L))$, 
which is realized by meridians $M_1$, \dots, $M_m$ positively linked 
to $K_1$, \dots, $K_m$, respectively. (The meridians are fibers of a
tubular fibration $\p X(L)\to L$ over points chosen on the corresponding
components.)
 
Therefore, if $L$ is oriented, then 
a local coefficient system on $X(L)$ with fiber $\C$ is defined 
by an $m$-tuple of complex numbers $(\Gz_1,\dots,\Gz_m)$, the images 
under the monodromy homomorphism
$H_1(X(L))\to\C^\times$ of the generators $[M_1]$, \dots, $[M_m]$ of 
$H_1(X(L))$.

Thus for an oriented classical knot $L$ consisting of $m$ connected
components, local coefficient systems 
on $X(L)$ with fiber $\C$ are parametrized by $(\C^\times)^m$.

\subsection{Link signatures}\label{s2.3} 
Let $L=K_1\cup\dots\cup K_m\subset S^3$ be a classical
link, $\Gz_i\in\C$, $|\Gz_i|=1$,
$\Gz=(\Gz_1,\dots,\Gz_m)\in (S^1)^m$ and $\mu:H_1(S^3\sminus
L)\to\C^\times$ takes to $\Gz_i$ a meridian of $K_i$ positively 
linked with $K_i$.

Let $F_1,\dots F_m\subset D^4$ be smooth oriented surfaces transversal 
to each other with $\p F_i=F_i\cap\p D^4=K_i$. Extend the tubular 
neighborhood of $L$ involved in the construction of $X(L)$ to a 
collection of tubular neighborhoods $N_1$, \dots, $N_m$ of 
$F_1$, \dots, $F_m$, respectively. 

Without loss of generality we may choose $N_i$ in such a way that they 
would intersect each other in the simplest way. 
Namely, each connected 
component $B$ of $N_i\cap N_j$ would contain only one point of 
$F_i\cap F_j$ and no point of others $F_k$ 
and would consist of entire fibers of $N_i$ and $N_j$, 
so that the fibers define a structure of bi-disk $D^2\times D^2$ on $B$.

To achieve this, one has to make the fibers of the tubular fibration 
$N_i\to F_i$ at each intersection point of $F_i$ and $F_j$ coinciding 
with a disk in $F_j$ and then diminish all $N_i$ appropriately. 

Now let us extend $X(L)$ to $X(F)=D^4\sminus\cup_{i=1}^m\Int N_i$.
This is a compact 4-manifold. Its boundary contains $X(L)$, the rest of
it is a union of pieces of boundaries of $N_i$ with $i=1,\dots, m$. 
These pieces are fibered over the corresponding pieces of $F_i$ with
fiber $S^1$.

By the Alexander duality, the orientation of $F_i$ gives rise to a 
homomorphism $H_1(X(F))\to\Z$ that maps a homology class to its 
linking number with $F_i$.
These homomorphisms altogether determine a homomorphism 
$H_1(X(F))\to \Z^m$. For any $\Gz=(\Gz_1,\dots,\Gz_m)$, the composition
of this homomorphism with the homomorphism 
$$\Z^n\to(\C^\times)^m:
(n_1,\dots,n_m)\to(\Gz_1^{n_1},\dots,\Gz_m^{n_m})$$     
is a homomorphism $H_1(X(F))\to(\C^\times)^m$ extending $\mu$.
 If each $F_i$ has no closed connected components, then this extension 
is unique. 
Let us denote it by $\overline\mu$. 

According to \ref{sT.4.6},
in $H_2(X(F);\C_{\overline\mu})$ there is a
Hermitian intersection form. Denote its signature by $\Gs_{\Gz}(L)$.

\begin{Th}  $\Gs_{\Gz}(L)$ 
does not depend on $F_1,\dots,F_m$.
\end{Th}

\begin{proof} Any $F'_i$ with $\p F'_i=F'_i\cap\p
D^4=K_i$ is cobordant to $F_i$.
The cobordisms $W_i\subset D^4\times I$ can be made 
pairwise transversal.
They define a cobordism $D^4\times I\sminus\cup_i\Int N(W_i)$
between $X(F)$ and
$X(F')$.
By Theorem \ref{VanSign},  
$$\Gs_\Gz(\p D^4\times I\sminus\cup_i\Int N(W_i))=0.$$
The manifold $\p D^4\times I\sminus\cup_i\Int N(W_i)$ is the union of
$X(F)$, $-X(F')$ and a {\it
homologically negligible\/} part 
$\p (N(\cup_i\Int W_i))$, 
the boundary of a regular neighborhood
of the cobordism $\cup_iW_i$ between $\cup_iF_i$ and $\cup_iF'_i$.
By Theorem \ref{AddOfSign}, 
$$\Gs_\Gz(\p D^4\times I\sminus\cup_i\Int
N(W_i))=\Gs_\Gz(D^4\sminus\cup_iF_i)-\Gs_\Gz(D^4\sminus\cup_iF'_i)$$
Hence, 
$\Gs_\Gz(D^4\sminus\cup_iF_i)=\Gs_\Gz(D^4\sminus\cup_iF'_i)$.
\end{proof}

\subsection{Colored links}\label{s2.4}  In the definition of signature 
$\Gs_{\Gz}(L)$ above one needs to numerate the components $K_i$ of $L$
to associate to each of them the corresponding component $\Gz_i$ of $\Gz$,
but there is no need to require connectedness of each $K_i$. 
This leads to a notion of colored link.  

 An {\sl $m$-colored link} $L$ is an oriented 
link in $S^3$ together with a map (called {\sl coloring\/}) assigning 
to each connected component of $L$ a color in $\{1,\dots, m\}$. 
The sublink $L_i$ is constituted by the components of $L$ with color $i$ 
for $i=1,\dots, m$.

For an $m$-colored link $L=L_1\cup\dots\cup L_m$ and 
$\Gz=(\Gz_1,\dots,\Gz_m)\in (S^1)^m$, the signature $\Gs_\Gz(L)$ is
defined as above, but each component $K_j$ colored with color $i$ is 
associated to $\Gz_i$.   

\subsection{Relations to other link signatures}\label{s2.5}

If $\Gz_i=-1$ for all $i=1,\dots,m$, then the signature $\Gs_{\Gz}(L)$ 
coincides
with the Murasugi signature $\xi(L)$ introduced in \cite{Mura2}. 
If all $\Gz_i$ are roots of unity of a degree, which is a power of a prime
number and all linking numbers $\lk(L_i,L_j)$ vanish, then $\Gs_{\Gz}(L)$ 
coincides with the signature defined by
Florens \cite{Florens1}.

In the most general case, $\Gs_{\Gz}(L)$ coincides with the signature 
defined for arbitrary $\Gz$ by Cimasoni and Florens \cite{CimaFlor}
using a 3-dimensional approach, with a version of Seifert surface, 
$C$-complex.   
                           
\section{In higher dimensions}\label{s3}

\subsection{Apology for the generalization of higher dimensional 
links}\label{s3.1} 

There is a spectrum of objects
considered as generalizations of classical knots and links.
The closest generalization of classical knots
are pairs $(S^n,K)$, where $K$ is a smooth submanifold
diffeomorphic to $S^{n-2}$. Then the requirements on $K$
are weakened. Say, one may require $K$ to be only homeomorphic 
to $S^{n-2}$, not diffeomorphic. Or just a homology sphere
of dimension $n-2$. The codimension is important in order to
keep any resemblance to classical knots.

In the same spirit, for the closest higher-dimensional 
counter-part of classical links one takes a pair consisting
of $S^n$ and a collection of its disjoint smooth
submanifolds diffeomorphic to $S^{n-2}$. One allows
to weaken the restrictions on the submanifolds. Up to arbitrary
closed submanifolds. 
\medskip

%\noindent
%\subsubsection{Intersecting components}\label{s3.1.1}
{\bfit I suggest to allow transversal intersections of the 
submanifolds.\/}
\medskip

Of course, the main excuse for this is that 
some results  can extended to this setup. 
Here is a couple of other reasons.

First, in the classical dimension, it is easy for submanifolds to be
disjoint. Generically curves in 3-sphere are disjoint. If 
they intersect, it is a miracle or, rather, has a special
cause. 

Generic submanifolds of codimension two in a manifold
of dimension $>3$ intersect. If they do {\sfit not\/} intersect, this
is a miracle, or consequence of a special cause. 

Second, classical links emerge naturally as links of 
singular points
of complex algebraic curves in $\C^2$. Recall that for an
algebraic curve $C\subset\C^2$ and a point $p\in C$ the
boundary of a sufficiently small ball $B$ centered at $p$, 
the link $(\p B,\p B\cap C)$ is  well-defined up to 
diffeomorphism, and it is called the {\sfit link of
$C$ at $p$}. 

An obvious generalization of this definition to an
algebraic hypersurface $C\subset\C^n$ gives rise to 
a pair $(S^{2n-1},K)$ with connected $K$. It cannot be
a union of {\it disjoint\/} submanifolds of $S^{2n-1}$.

It would not be difficult to extend the results of this paper 
to a more general setup. For example, one can replace the ambient
sphere with a homology sphere, or even more general manifold.
However, one should stop somewhere. The author prefers this early
point, because the level of generality accepted here suffices for
demonstrating the new opportunities open by a systematic usage of
twisted homology. On the other hand, further generalizations 
can make formulations more cumbersome.

\subsection{Colored links}\label{s3.2}
By an {\sl $m$-colored link of dimension\/} $n$ we shall mean a collection
of $m$ oriented smooth closed $n$-dimensional submanifolds $L_1$, \dots, 
$L_m$ of the sphere $S^{n+2}$ such that any sub-collection has 
transversal intersection. The latter means that for any 
$x\in L_{i_1}\cap\dots\cap L_{i_k}$ the tangent spaces $T_xL_{i_1}$, \dots,
$T_xL_{i_k}$ are transverse, that is, $\dim(T_xL_{i_1}\cap\dots\cap
T_xL_{i_k})=n+2-2k$. 

\subsection{Generic configurations of submanifolds}\label{s3.3}
More generally, an $m$-colo\-red configuration of transversal submanifolds
in a smooth manifold $M$ is a family of $m$  smooth
submanifolds $L_1$, \dots, $L_m$ of $M$ such that any sub-collection has
transversal intersection. If $M$ has a boundary, the submanifolds are
assumed to be transversal to the boundary, as well as the intersection of
any sub-collection. Furthermore, assume that $\p M\cap L_i=\p L_i$ for any
$i=1,\dots,m$.   

As above, in Section \ref{s2.3}, for any $m$-colored configuration $L$ 
of transversal submanifolds $L_1$, \dots, $L_m$ in $M$ one can find a
collection of their tubular neighborhoods $N_1$, \dots, $N_m$ which agree
with each other in the sense that for any sub-collection $L_{i_1}$, \dots,
$L_{i_\nu}$ the intersection of the corresponding neighborhoods 
$N_{i_1}\cap\dots\cap N_{i_\nu}$ is neighborhood of the intersection 
$L_{i_1}\cap\dots\cap L_{i_\nu}$ fibered over this intersection with 
the corresponding poly-disk fiber. 

Denote the complement $M\sminus\cup_{i=1}^m\Int N_i$ by $X(L)$ and call it
an {\sfit exterior\/} of $L$. This is a smooth manifold with a system of
corners on the boundary.  The differential type of the exterior does
not depend on the choice of neighborhoods. Moreover, one can eliminate the
choice of neighborhoods and deleting of them from the definition. Instead,
one can make a sort of real blowing up of $M$ along $L_1$, \dots, $L_m$.
However, for the purposes of this paper it is easier to stay with the
choices. 
 
\subsection{Link signatures}\label{s3.4}
Let $L=L_1\cup\dots\cup L_m$ be an $m$-colored link of dimension $2n-1$ in
$S^{2n+1}$. 

As well known (see, e.g., \cite{Levine1}), for each oriented closed 
codimension 2 submanifold $K$ of $S^{2n+1}$ 
there exists an oriented smooth compact submanifold 
$F$ of $D^{2n+2}$ such that $\p F=K$. Choose for each $L_i$ such a
submanifold of $D^{2n+2}$, denote it by $F_i$, and make all the $F_i$ 
transversal to each other by small perturbations. 

As a union of $m$-colored transversal submanifolds of $D^{2n+2}$,
$F= F_1\cup \dots\cup F_m$ has an exterior $X(F)$. 
By the Alexander duality, $H^1(X(F);\C^\times)$ is naturally isomorphic to
$H_{2n}(F,L;\C^\times)$. Let $\Gz=(\Gz_1,\dots,\Gz_m)\in(S^1)^m$.
Take $\sum_{i=1}^m\Gz_i[F_i]\in H_{2n}(F,L;\C^\times)$ and denote
by $\mu$ the Alexander dual cohomology class considered as a homomorphism
$H_1(X(F))\to\C^\times$. Denote by $\C_\mu$ the local coefficient system 
on $X(F)$ corresponding to $\mu$.

According to \ref{sT.4.6}, in $H_{n+1}(X(F);\C_\mu)$ there is an 
intersection form, which is Hermitian, if $n$ is odd, and skew-Hermitian,
if $n$ is even. Denote its signature by $\Gs_\Gz(L)$.

\begin{Th}  $\Gs_{\Gz}(L)$ 
does not depend on $F_1,\dots,F_m$.
\end{Th}

\begin{proof} Any $F'_i$ with $\p F'_i=F'_i\cap\p
D^{2n+2}=L_i$ is cobordant to $F_i$.
The cobordisms $W_i\subset D^{2n+2}\times I$ can be made 
pairwise transversal to form $m$-colored configuration $W$ of 
transversal submanifolds of $D^{2n+2}\times I$.
They define a cobordism $X(W)$
between $X(F)$ and
$X(F')$.
By Theorem \ref{VanSign},  
$$\Gs_\Gz(\p X(W))=0.$$
The manifold $\p X(W)=\p D^{2n+2}\times I\sminus\cup_i\Int N(W_i)$ 
is the union of
$X(F)$, $-X(F')$ and a {\it
homologically negligible\/} part 
$\p (N(\cup_i\Int W_i))$, 
the boundary of a regular neighborhood
of the cobordism $\cup_iW_i$ between $F$ and $F'$.
By Theorem \ref{AddOfSign}, 
$$\Gs_\Gz(\p X(W))=\Gs_\Gz(X(F))-\Gs_\Gz(X(F'))$$
Hence, 
$\Gs_\Gz(X(F))=\Gs_\Gz(X(F'))$.
\end{proof}

\section{Span inequalities}\label{s4} 
Let $L=L_1\cup\dots,\cup L_m$ be an $m$-colored link of dimension $2n-1$ 
in $S^{2n+1}$. Let $F=F_1\cup\dots\cup F_m$ be an $m$-colored 
configuration of transversal oriented compact $2n$-dimensional 
submanifolds of $D^{2n+2}$ with $\p F_i=F_i\cap\p D^{2n+2}=L_i$. 
In this section we consider restrictions on homological
characteristics of $F$ in terms of invariants of $L$.

\subsection{History}\label{s4.1}
The first restrictions of this sort were found by Murasugi \cite{Mura1} 
and Tristram \cite{Trist} for classical (1-colored) links. To $m$-colored
classical links and pairwise disjoint surfaces $F_i$ the Murasugi-Tristram
inequalities were generalized by Florens \cite{Florens1}. A further
generalization to $m$-colored classical links and intersecting $F_i$
was found by Cimasoni and Florens \cite{CimaFlor}. Higher dimensional
generalizations for $1$-colored links were found by the author
\cite{Viro2}, \cite{Viro3}.

\subsection{No-nullity span inequalities}\label{s4.2}
The most general results in this direction are quite cumbersome. 
Therefore, let me start with weak but simple ones.

Recall that $\Gs_\Gz(L)$ can be obtained from $F$: for an appropriate 
local coefficient system $\C_\mu$ on $X(F)$, this is the signature
of a Hermitian intersection form defined in $H_{n+1}(X(F);\C_\mu)$.
The signature of an Hermitian form cannot be greater than the 
dimension of the underlying space. In particular, 
\begin{equation}\label{ineq:S=<D}
|\Gs_\Gz(L)|\le\dim_\C H_{n+1}(X(F);\C_\mu).
\end{equation}

This can be considered as a restriction on a homological
characteristic of $F$ in terms of invariants of $L$. However, 
$\dim_\C H_{n+1}(X(F);\C_\mu)$ is not a convenient characteristic of $F$.
It can be estimated in terms of more convenient ones.

Let $\Gz=(\Gz_1,\dots,\Gz_m)\in (S^1)^m$. 
Let $p_1,\dots,p_k\in\Z[t_1,t_1^{-1}\dots,t_m,t_m^{-1}]$ be 
generators of the ideal of relations 
satisfied by complex numbers $\Gz_i$. Let $d$ be the
greatest common divisor of the integers 
$p_1(1,\dots,1)$, \dots, $p_k(1,\dots,1)$, if at least one of these
integers does not vanish, and zero otherwise. Cf. \ref{sT.3.6}. Let 
$$
P=\begin{cases} \Z/p\Z, &\text{ if }d>1\text{ and }p\text{ is a prime
divisor of }d\\
\Q, &\text{ if }d=0 
\end{cases}
$$

By \ref{EstTwHom}, 
$$\dim_\C H_{n+1}(X(F);\C_\mu)\le
\dim_{P} H_{n+1}(X(F);P).
$$
The advantage of passing to homology with non-twisted coefficients is 
that we can use the Alexander duality: \begin{multline*}
H_{n+1}(X(F);P)=H_{n+1}(D^{2n+2}\sminus F;P)\\=
H^{n+1}(D^{2n+2},\p D^{2n+2}\cup F;P)\\=
H^{n}(\p D^{2n+2}\cup F;P)=
H^n(F,L;P).
\end{multline*}
Hence,
$$|\Gs_\Gz(L)|\le \dim_{P}H_n(F,L;P).
$$ 

\subsection{General span inequalities}\label{s4.3}
The inequality \eqref{ineq:S=<D} can be improved. Indeed, the manifold 
$X(F)$ has a non-empty boundary. Therefore, its intersection form may be
degenerate and the right hand side of \eqref{ineq:S=<D} may be replaced by
a smaller quantity, the rank of the form. The rank is known to be
the rank of the homomorphism 
$H_{n+1}(X(F);\C_\mu)\to H_{n+1}(X(F),\p X(F);\C_\mu)$. 
Let us estimate this rank.

\begin{lem}\label{Lemma1Th7}
For any exact sequence $\dots\overset{\rho_{k+1}}\to C_k\overset{\rho_k}\to
 C_{k-1}\overset{\rho_{k-1}}\to\dots$ of vector spaces
and any integers $n$ and $r$ 
\begin{equation}\label{eqL1Th7} \rnk(\rho_{n+1})+\rnk(\rho_{n-2r})
=\sum_{s=0}^{2r}(-1)^s\dim C_{n-s}
\end{equation}
\end{lem}

\begin{proof}The Euler characteristic of the exact sequence
$$
0\to\Im\rho_{n+1}\hookrightarrow C_n\overset{\rho_{n}}\to
C_{n-1}\to\dots\overset{\rho_{n-2r+1}}\to C_{n-2r}\to\Im\rho_{n-2r}\to0
$$ 
is the difference between the left and right hand sides of \eqref{eqL1Th7}.
On the other hand, it vanishes, as the Euler characteristic of an exact
sequence. 
\end{proof}

\begin{lem}\label{Lemma2Th7}
Let $X$ be a topological space, $A$ its subspace, $\xi$ a local coefficient
system on $X$ with fiber $\C$. Then for any natural $n$ and $r\le\frac{n}2$
\begin{multline}
\rnk(H_{n+1}(X;\xi)\to H_{n+1}(X,A;\xi))+
 \rnk(H_{n-2r}(X;\xi)\to H_{n-2r}(X,A;\xi))\\
\\
=\sum_{s=0}^{2r}(-1)^sb_{n+1-s}(X,A)
-\sum_{s=0}^{2r}(-1)^sb_{n-s}(A)
+\sum_{s=0}^{2r}(-1)^sb_{n-s}(X)
\end{multline} 
where $b_k(*)=\dim_\C H_k(*;\xi)$
\end{lem}

\begin{proof}
Apply Lemma \ref{Lemma1Th7} to the homology sequence of pair $(X,A)$ with
coefficients in $\xi$.
\end{proof}

\begin{Th}\label{Th7} For any integer $r$ with $0\le r\le\frac{n}2$\\   
\begin{multline}\label{ineq:Span}
|\Gs_{\Gz}(L)|+\sum_{s=0}^{2r}(-1)^s\dim_\C
  H_{n-s}(S^{2n+1}\sminus L;\C_\Gz)
\\  \le \sum_{s=0}^{2r}(-1)^s\dim H_{n+1+s}(F,L;P)
  +\sum_{s=0}^{2r}(-1)^s\dim H_{n+s}(F;P)
\end{multline}
\begin{multline}\label{ineq:Span2}
|\Gs_{\Gz}(L)|+\sum_{s=0}^{2r}(-1)^s\dim_\C
  H_{n+1+s}(S^{2n+1}\sminus L;\C_\Gz)
\\  \le \sum_{s=0}^{2r}(-1)^s\dim H_{n-s}(F,L;P)
  +\sum_{s=0}^{2r}(-1)^s\dim H_{n-s-1}(F;P)
\end{multline}
where $\Gz$ and $P$ are is in Section \ref{s4.2}
\end{Th}

\begin{proof}As mentioned above, 
\begin{equation}\label{eq1PfTh7}
|\Gs_\Gz(L)|\le \rnk(H_{n+1}(X(F);\C_\mu)\to H_{n+1}(X(F),\p X(F);\C_\mu)).
\end{equation}
By Lemma \ref{Lemma2Th7}, 
\begin{multline}\label{eq2PfTh7}
\rnk(H_{n+1}(X(F);\C_\mu)\to H_{n+1}(X(F),\p X(F);\C_\mu))\\
\le \sum_{s=0}^{2r}(-1)^s\dim_\C H_{n+1-s}(X(F),X(L);\C_\Gz)
-\sum_{s=0}^{2r}(-1)^s\dim_\C H_{n-s}(X(L);\C_\Gz)\\
+\sum_{s=0}^{2r}(-1)^s\dim_\C H_{n-s}(X(F);\C_\Gz).
\end{multline}
Summing up these inequalities and moving one of the sums from the right
hand side to the left, we obtain:
\begin{multline}\label{eq3PfTh7}
|\Gs_\Gz(L)|+\sum_{s=0}^{2r}(-1)^s\dim_\C H_{n-s}(X(L);\C_\Gz)\\
\le \sum_{s=0}^{2r}(-1)^s\dim_\C H_{n+1-s}(X(F),X(L);\C_\Gz)
+\sum_{s=0}^{2r}(-1)^s\dim_\C H_{n-s}(X(F);\C_\Gz).
\end{multline}
The left hand sum of \eqref{eq3PfTh7} coincides with the left hand side
of \eqref{ineq:Span}. The right hand side can be estimated using 
Theorem \ref{EstTwHom}:
\begin{multline}\label{eq4PfTh7}
\sum_{s=0}^{2r}(-1)^s\dim_\C H_{n+1-s}(X(F),X(L);\C_\Gz)
+\sum_{s=0}^{2r}(-1)^s\dim_\C H_{n-s}(X(F);\C_\Gz)\\
\le\sum_{s=0}^{2r}(-1)^s\dim_P H_{n+1-s}(X(F),X(L);P)
+\sum_{s=0}^{2r}(-1)^s\dim_P H_{n-s}(X(F);P).
\end{multline}
Further, 
$$H_{n+1-s}(X(F),X(L);P)=H_{n+1-s}(D^{2n+2}\sminus F,S^{2n+1}\sminus L;P).
$$
By the Alexander duality,
$$ H_{n+1-s}(D^{2n+2}\sminus F,S^{2n+1}\sminus L;P)
  =
H^{n+1+s}(D^{2n+2},F;P).
$$
By exactness of the pair sequence, $H^{n+1+s}(D^{2n+2},F;P)=H^{n+s}(F;P)$.

Similarly,
\begin{multline*}
H_{n-s}(X(F);P)=H_{n-s}(D^{2n+2}\sminus F;P)\\
=H^{n+2+s}(D^{2n+2},F\cup S^{2n+1};P)\\
=H^{n+1+s}(S^{2n+1}\cup F;P)=H^{n+1+s}(F,L;P)
\end{multline*}
The last equality in this sequence holds true if $n+1+s<2n+1$, that is,
$s<n$. 

Since $P$ is a field, 
\begin{align}
\dim_P H^{n+s}(F;P)=&\dim_P H_{n+s}(F;P),\label{eq5PfTh7} \\
\dim_P H^{n+1+s}(F,L;P)=&\dim_P H_{n+1+s}(F,L;P)\label{eq6PfTh7}.
\end{align}

Combining formulas \eqref{eq5PfTh7}, \eqref{eq6PfTh7} with the calculations
above and equalities \eqref{eq4PfTh7} and \eqref{eq3PfTh7}, we obtain the
first desired inequalities \eqref{ineq:Span}.

The inequalities \eqref{ineq:Span2} are proved similarly. Namely,
by Lemma \ref{Lemma2Th7} 
\begin{multline}\label{eq7PfTh7}
\rnk(H_{n+1}(X(F);\C_\mu)\to H_{n+1}(X(F),\p X(F);\C_\mu))\\
\le \sum_{s=0}^{2r}(-1)^s\dim_\C H_{n+2+s}(X(F),X(L);\C_\Gz)
-\sum_{s=0}^{2r}(-1)^s\dim_\C H_{n+1+s}(X(L);\C_\Gz)\\
+\sum_{s=0}^{2r}(-1)^s\dim_\C H_{n+1+s}(X(F);\C_\Gz).
\end{multline}
Summing up inequalities \eqref{eq1PfTh7} and \eqref{eq7PfTh7} and 
moving one of the sums from the right
hand side to the left, we obtain:
\begin{multline}\label{eq8PfTh7}
|\Gs_\Gz(L)|+\sum_{s=0}^{2r}(-1)^s\dim_\C H_{n+1+s}(X(L);\C_\Gz)\\
\le \sum_{s=0}^{2r}(-1)^s\dim_\C H_{n+2+s}(X(F),X(L);\C_\Gz)
+\sum_{s=0}^{2r}(-1)^s\dim_\C H_{n+1+s}(X(F);\C_\Gz).
\end{multline}
After this the same estimates and transformations as in the proof of 
\eqref{ineq:Span} gives rise to \eqref{ineq:Span2}.
\end{proof}

\subsection{Nullities}\label{s4.4}

The sum in the left hand side of the inequalities \eqref{ineq:Span} is an
invariant of the link $L$. Its special case for classical links with $r=0$ 
is known as $\Gz$-nullity and appeared in the Murasugi-Tristram inequalities and their generalizations. 

Denote  $\sum_{s=0}^{2r}(-1)^s\dim
H_{n-s}(S^{2n+1}\sminus L;\C_\mu)$ by  $n^r_{\Gz}(L)$ and call it 
{\sfit $r$th $\Gz$-nullity of $L$\/}.

By the Poincar\'e duality (see \ref{sT.4.3}), 
$H_{n-s}(S^{2n+1}\sminus L;\C_\mu)$ is isomorphic to 
$H^{n+1+s}(S^{2n+1}\sminus L;\C_\mu)$. The latter vector space 
is dual to $H_{n+1+s}(S^{2n+1}\sminus L;\C_{\mu^{-1}})$ and anti-isomorphic
to  $H_{n+1+s}(S^{2n+1}\sminus L;\C_{\mu})$, see \ref{sT.4.5}. 
Therefore,
\begin{equation}\label{null}
n^r_{\Gz}(L)=\sum_{s=0}^{2r}(-1)^s\dim_\C H_{n+1+s}(S^{2n+1}\sminus
L;\C_\mu)
\end{equation}
and $n^r_\Gz(L)=n^r_{\overline{\Gz}}(L)$.
This sum is a part of the left hand side of \eqref{ineq:Span2}.

Now we can rewrite Theorem \ref{Th7} as follows:

\begin{Th}\label{Th8} For any integer $r$ with $0\le 2r\le n$
\begin{multline}\label{null2}
|\Gs_{\Gz}(L)|+n^r_\Gz(L)
 \\  \le \sum_{s=0}^{2r}(-1)^s\dim H_{n+s+1}(F,L;P)
 +\sum_{s=0}^{2r}(-1)^s\dim H_{n+s}(F;P)
\end{multline}

\begin{multline}\label{null3}
|\Gs_{\Gz}(L)|+n^r_\Gz(L)
 \\  \le \sum_{s=0}^{2r}(-1)^s\dim H_{n-s}(F,L;P)
 +\sum_{s=0}^{2r}(-1)^s\dim H_{n-s-1}(F;P)
\end{multline}
\end{Th}

If $F_i$ are pairwise disjoint, than the right hand sides of \eqref{null2}
and \eqref{null3} are equal due to Poincar\'e-Lefschetz duality
for $F$, but we do not assume that $F=\cup F_i$ is a manifold, and
therefore the inequalities \eqref{null2} and \eqref{null3} are not
equivalent and we have to keep both of them. 

%Special cases:
% $$|\Gs_\Gz(L)|+n^0_\Gz(L) 
%\le\dim H_n(F;P)+\dim H_{n+1}(F,L;P).$$

\section{Slice inequalities}\label{s5} 
Again, as in the preceding section, let $L_1,\dots, L_m\subset S^{2n+1}$ 
be smooth  oriented transversal to each other submanifolds
constituting an $m$-colored link $L=L_1\cup\dots\cup L_m$ of 
dimension $2n-1$. 

Let $\GL_i\subset S^{2n+2}$ be oriented closed smooth submanifolds 
transversal to each other and to $S^{2n+1}$, 
with $\p \GL_i\cap S^{2n+1}=L_i$. In this section we consider restrictions
on homological characteristics of $\GL=\cup_{i=1}^m\GL_i$ in terms of 
invariants of link $L$. Of course, some results of this kind can be deduced
from the results of the preceding section, but an independent consideration
gives better results.

\subsection{No-nullity slice inequalities}\label{s5.1}
The most general results in this direction are quite cumbersome. 
Therefore, let me start with weak but simple ones.

We will use the same algebraic objects as in the preceding
section. In particular, $\Gz=(\Gz_1,\dots,\Gz_m)\in (S^1)^m$,
 $p_1,\dots,p_k\in\Z[t_1,t_1^{-1}\dots,t_m,t_m^{-1}]$ are
generators of the ideal of relations 
satisfied by complex numbers $\Gz_i$. Integer $d$ is the
greatest common divisor of the integers 
$p_1(1,\dots,1)$, \dots, $p_k(1,\dots,1)$, if at least one of them
does not vanish, and $d=0$ otherwise. Cf. \ref{s4.2} and \ref{sT.3.6}. 
Finally, 
$$
P=\begin{cases} \Z/p\Z, &\text{ if }d>1\text{ and }p\text{ is a prime
divisor of }d\\
\Q, &\text{ if }d=0 
\end{cases}
$$

Let $\mu:H_1(S^{2n+1}\sminus L)\to\C^\times$ be the homomorphism which
maps the meridian of $L_i$ to $\Gz_i$. 
The local coefficient system $\C_\mu$ on $S^{2n+1}\sminus L$ defined by
$\mu$  extends to $S^{2n+2}\sminus\GL$. 
We will denote the extension by the same symbol $\C_\mu$.

The sphere $S^{2n+1}$ bounds in $S^{2n+2}$ two balls, hemi-spheres 
$S^{2n+2}_+$ and $S^{2n+2}_-$ such that $\p S^{2n+2}_+=S^{2n+1}$ 
and $\p S^{2n+2}_-=-S^{2n+1}$ with the orientations inherited
from the standard orientation of $S^{2n+2}$. In
$H_{n+1}(S^{2n+2}\sminus\GL;\C_\mu)$ there is a (Hermitian or
skew-Hermitian) intersection form. Its signature is zero by Theorem
\ref{VanSign}, because $\GL$ bounds a configuration of pairwise 
transversal submanifolds $\GD=\GD_1\cup\dots\cup\GD_m$ in $D^{2n+3}$ 
and $\C_\mu$ 
extends over $D^{2n+3}\sminus\GD$. 

\begin{Th}\label{slice-small-Th} Under the assumption above,
\begin{equation}\label{ineq:Slice-easy}
2|\Gs_\Gz(L)|\le\dim_P H_n(\GL;P).
\end{equation}
\end{Th}

\begin{proof}
The intersection form on $H_{n+1}(S^{2n+2}\sminus\GL;\C_\mu)$ restricted
to the images of $H_{n+1}(S^{2n+2}_+\sminus\GL;\C_\mu)$ and  
$H_{n+1}(S^{2n+2}_-\sminus\GL;\C_\mu)$ has signatures $\Gs_\Gz(L)$ and
$-\Gs_\Gz(L)$, respectively. Therefore the dimension of each of the images
is at least $|\Gs_\Gz(L)|$. 

The images are obviously orthogonal to each
other with respect to the intersection form, because their elements 
can be realized by cycles lying in disjoin open hemi-spheres.
Hence 
$$
2|\Gs_\Gz(L)|\le\dim_\C H_{n+1}(S^{2n+2}\sminus\GL;\C_\mu).
$$
On the other hand, by Theorem \ref{EstTwHom}, 
$$\dim_\C H_{n+1}(S^{2n+2}\sminus\GL;\C_\mu)\le 
\dim_P H_{n+1}(S^{2n+2}\sminus\GL;P)=\dim_P H_n(\GL;P).$$ 
Summing up these two inequalities, we obtain the desired one.
\end{proof}

\subsubsection{General slice inequalities}\label{s5.2}
\begin{Th}\label{Th-Slice}
Under assumptions above
\begin{multline}\label{ineq:slice} 
2|\Gs_\Gz(L)| +2n^r_\Gz(L)\\
\le 
\sum_{s=0}^{2r}(-1)^s\dim_P H_{n-s}(\GL\sminus L;P)
+\sum_{s=-2r+1}^{2r-1}(-1)^s\dim_P H_{n-s}(\GL;P)
\end{multline}
\end{Th}

\begin{lem}\label{Lemma1Th} 
Let $j$ be the inclusion $S^{2n+1}\sminus L\to S^{2n+2}\sminus\GL$. 
Then
\begin{multline}\label{eq0PfTh} 
2|\Gs_\Gz(L)| +
2\rnk(j_*:H_{n+1}(S^{2n+1}\sminus L;\C_\mu)\to 
H_{n+1}(S^{2n+2}\sminus\GL;\C_\mu))\\
\le \dim_\C H_{n+1}(S^{2n+2}\sminus\GL;\C_\mu)
\end{multline}
\end{lem}
\begin{proof}
Denote by $i^\pm$ the inclusion 
$S^{2n+2}_\pm\sminus\GL\to S^{2n+2}\sminus\GL$.
Observe that the space $H_{n+1}(S^{2n+2}\sminus\GL;\C_\mu)$ 
has a natural filtration:
\begin{multline}\label{eq1PfTh}
j_*H_{n+1}(S^{2n+1}\sminus L;\C_\mu)\\
\subset
i^+_*H_{n+1}(S^{2n+2}_+\sminus\GL;\C_\mu)+
i^-_*H_{n+1}(S^{2n+2}_-\sminus\GL;\C_\mu)\\
\subset
H_{n+1}(S^{2n+2}\sminus\GL;\C_\mu)
\end{multline}
The inclusion homomorphisms 
$$j_*:H_{n+1}(S^{2n+1}\sminus L;\C_\mu)\to 
H_{n+1}(S^{2n+2}\sminus\GL;\C_\mu)$$ 
and the boundary homomorphism 
$$
\p:H_{n+1}(S^{2n+2}\sminus\GL;\C_\mu)\to H_{n}(S^{2n+1}\sminus L;\C_\mu) 
$$
of the Mayer-Vietoris sequence of the triad $(S^{2n+2}\sminus\GL;
S^{2n+2}_+\sminus\GL,S^{2n+2}_-\sminus\GL)$ are dual to each other
with respect to the intersection forms: 
$$
j_*(a)\circ b=a\circ\p(b)\ \text{ for any }a\in
H_{n+1}(S^{2n+1}\sminus L;\C_\mu)\text{ and }b\in
H_{n+1}(S^{2n+2}\sminus\GL;\C_\mu).
$$
Since the intersection forms are non-singular, it follows that 
$\rnk j_*=\rnk \p$.

By exactness of the Mayer-Vietoris sequence, the rank of $\p$ is 
the dimensions of the top quotient of the filtration \eqref{eq1PfTh}, while
the rank of $j_*$ is the dimension of the smallest term 
$j_*H_{n+1}(S^{2n+1}\sminus L;\C_\mu)$ of this filtration.

The middle term of the filtration contains the subspaces 
$i^+_*H_{n+1}(S^{2n+2}_+\sminus\GL;\C_\mu)$ and 
$i^-_*H_{n+1}(S^{2n+2}_-\sminus\GL;\C_\mu)$. Their intersection 
is the smallest term, which is orthogonal to both of the subspaces.
Therefore the dimension of the quotient of the middle term of the
filtration by the smallest term is at least $2|\Gs_\Gz(L)|$

The dimension of the whole space $H_{n+1}(S^{2n+2}\sminus\GL;\C_\mu)$
is the sum of the dimensions of the factors. We showed above that the
top and lowest factor have the same dimensions equal to $\rnk j_*$ and 
that the dimension of the middle factor is at least $2|\Gs_\Gz(L)|$.
\end{proof}

\begin{lem}\label{Lemma1Th8}
For any exact sequence $\dots\overset{\rho_{k+1}}\to C_k\overset{\rho_k}\to
 C_{k-1}\overset{\rho_{k-1}}\to\dots$ of vector spaces
and any integers $n$ and $t$ 
\begin{equation}\label{eqL1Th8} \rnk(\rho_{n})-\rnk(\rho_{n+2t})
=\sum_{s=0}^{2t-1}(-1)^s\dim C_{n+s}
\end{equation}
\end{lem}

\begin{proof}The Euler characteristic of the exact sequence
$$
0\to\Im\rho_{n+2t}\hookrightarrow C_{n+2t-1}\overset{\rho_{n+2t-1}}\to
C_{n+2t-2}\to\dots\overset{\rho_{n+1}}\to C_{n}\to\Im\rho_{n}\to0
$$ 
is 
$
\rnk(\rho_{n})-\sum_{s=0}^{2t-1}(-1)^s\dim C_{n+s}-\rnk(\rho_n+2t)$,
that is the difference between the left and right hand sides of
\eqref{eqL1Th8}.
On the other hand, it vanishes, as the Euler characteristic of an exact
sequence. 
\end{proof}

\begin{lem}\label{Lemma2Th8}
Let $X$ be a topological space, $A$ its subspace, $\xi$ a local coefficient
system on $X$ with fiber $\C$. Then for any natural $n$ and integer 
$r$
\begin{multline}
\rnk(H_{n+1}(A;\xi)\to H_{n+1}(X;\xi))-
 \rnk(H_{n+2+2r}(X;\xi)\to H_{n+2+2r}(X,A;\xi))\\
\\=
\sum_{s=0}^{2r}(-1)^sb_{n+1+s}(A)
-\sum_{s=0}^{2r}(-1)^sb_{n+2+s}(X,A)
+\sum_{s=0}^{2r-1}(-1)^sb_{n+2+s}(X)
\end{multline} 
where $b_k(*)=\dim_\C H_k(*;\xi)$.
\end{lem}

\begin{proof}
Apply Lemma \ref{Lemma1Th8} to the homology sequence of pair $(X,A)$ with
coefficients in $\xi$.
\end{proof}

\begin{lem}\label{Lemma2Th}
For any integer $r$ with $0\le r\le\frac{n}2$\\   
\begin{multline}\label{ineq:Slice}
2|\Gs_\Gz(L)|+2n^r_\Gz(L)\\
\le 2\sum_{s=0}^{2r}(-1)^s\dim_\C
H_{n+2+s}(S^{2n+2}\sminus\GL,S^{2n+1}\sminus L;\C_\mu)\\
+\sum_{s=-2r+1}^{2r-1}(-1)^s\dim_\C H_{n+1+s}(S^{2n+2}\sminus\GL;\C_\mu)
\end{multline}
\end{lem}

\begin{proof}By Lemma \ref{Lemma2Th8} 
applied to the pair
$(S^{2n+2}\sminus\GL,S^{2n+1}\sminus L)$, we obtain
\begin{multline}\rnk(j_*:H_{n+1}(S^{2n+1}\sminus L;\C_\mu)
\to H_{n+1}(S^{2n+2}\sminus\GL;\C_\mu))\\
\ge 
\sum_{s=0}^{2r}(-1)^sH_{n+1+s}(S^{2n+1}\sminus L;\C_\mu)\\
-\sum_{s=0}^{2r}(-1)^s\dim_\C
H_{n+2+s}(S^{2n+2}\sminus\GL,S^{2n+1}\sminus L;\C_\mu)\\
+\sum_{s=0}^{2r-1}(-1)^s\dim_\C H_{n+2+s}(S^{2n+2}\sminus\GL;\C_\mu)
\end{multline}
From this inequality and inequality \eqref{eq0PfTh} we obtain
\begin{multline}\label{eq2PfTh}
2|\Gs_\Gz(L)|+2n^r_\Gz(L)\\
\le 2\sum_{s=0}^{2r}(-1)^s\dim_\C
H_{n+1+s}(S^{2n+2}\sminus\GL,S^{2n+1}\sminus L;\C_\mu)\\
-2\sum_{s=0}^{2r-1}(-1)^s\dim_\C H_{n+s+2}(S^{2n+2}\sminus\GL;\C_\mu)\\
+\dim_\C H_{n+1}(S^{2n+2}\sminus\GL;\C_\mu)
\end{multline}
From this and the Alexander duality (which states that
$H_{n+1+s}(S^{2n+2}\sminus\GL;\C_\mu)$ is isomorphic to
$H_{n+1-s}(S^{2n+2}\sminus\GL;\C_{\mu})$) 
the desired inequality follows. 
\end{proof}

\begin{lem}\label{Lemma3Th}
\begin{multline}
\sum_{s=0}^{2r}(-1)^s
\dim_\C H_{n+1+s}(S^{2n+2}\sminus\GL,S^{2n+1}\sminus L;\C_\mu)\\
\le
\sum_{s=0}^{2r}(-1)^s\dim_P H_{n-s}(\GL\sminus L;P)
\end{multline}
\end{lem}

\begin{proof}
By Theorem \ref{EstTwHom}
\begin{multline}
\sum_{s=0}^{2r}(-1)^s
\dim_\C H_{n+1+s}(S^{2n+2}\sminus\GL,S^{2n+1}\sminus L;\C_\mu)\\
\le
\sum_{s=0}^{2r}(-1)^s
\dim_P H_{n+1+s}(S^{2n+2}\sminus\GL,S^{2n+1}\sminus L;P).
\end{multline}
By Poincar\'e duality (cf. \ref{sT.4.3}),
$H_{n+1+s}(S^{2n+2}\sminus\GL,S^{2n+1}\sminus L;P)$
is isomorphic to $H^{n+1-s}(S^{2n+2}\sminus S^{2n+1},\GL\sminus L;P)$.
The latter is isomorphic to $H^{n-s}(\GL\sminus L;P)$. By the universal
coefficients formula, $H^{n-s}(\GL\sminus L;P)$ is isomorphic to 
$H_{n-s}(\GL\sminus L;P)$.
\end{proof}

\begin{lem}\label{Lemma4Th}
\begin{multline} 
\sum_{s=-2r+1}^{2r-1}(-1)^s\dim_\C H_{n+1+s}(S^{2n+2}\sminus\GL;\C_\mu)\\
\le \sum_{s=-2r+1}^{2r-1}(-1)^s\dim_P H_{n-s}(\GL;P)
\end{multline}
\end{lem}

\begin{proof} 
By Theorem \ref{EstTwHom}
\begin{multline} 
\sum_{s=-2r+1}^{2r-1}(-1)^s\dim_\C H_{n+1+s}(S^{2n+2}\sminus\GL;\C_\mu)\\
\le \sum_{s=-2r+1}^{2r-1}(-1)^s\dim_P H_{n+1+s}(S^{2n+2}\sminus\GL;P).
\end{multline}
By Poincar\'e duality,
$H_{n+1+s}(S^{2n+2}\sminus\GL;P)$ is isomorphic to 
$H^{n+1-s}(S^{2n+2},\GL;P)$.
From the  sequence of pair $(S^{2n+2},\GL)$ it follows that 
$H^{n+1-s}(S^{2n+2},\GL;P)$ is isomorphic 
to $H^{n-s}(\GL;P)$. By the universal coefficient formula, 
$H^{n-s}(\GL;P)$ is isomorphic
to $H_{n-s}(\GL;P)$. 
\end{proof}
\medskip

\noindent
{\bf Proof of Theorem \ref{Th-Slice}.} 
Sum up the inequalities of the last three Lemmas.\qed

%\bigskip

%\medskip

\appendix
  \renewcommand\thesection{\appendixname}%\ \Alph{section}}%
  \renewcommand\thesubsection{\appendixname\ \Alph{subsection}}%
  \renewcommand\theequation{\Alph{subsection}.\arabic{equation}} %

%\centerline{\bf Appendices. 
\section{Twisted homology}

\vspace{10pt}

\subsection{Twisted coefficients and chains}\label{sT.1}%{s1-1}
%\vspace{6pt}

\subsubsection{Local coefficient system}\label{sT.1.1} %{s1.1}
Let $X$ be a topological space, and $\xi$ be a $\C$-bundle over $X$
with a fixed flat connection. 

Here by a {\sfit connection\/} we mean operations of
{\sfit parallel transport\/}: for any path $s$ in $X$ connecting 
points $x$ and
$y$ the parallel transport $T_s$ is an isomorphism from the fiber $\C_x$
over $x$ to the fiber $\C_y$ over $y$, such that the parallel transport
along product of paths equals the composition of parallel transports along the
factors. In formula: $T_{uv}=T_v\circ T_u$. A connection is flat, if
the parallel transport isomorphism does not change when the path is
replaced by a homotopic path.

A flat connection in a bundle $\xi$ over a simply connected $X$ gives a
trivialization of $\xi$. 

Another name for $\xi$ is a {\sfit local
coefficient system\/} with fiber $\C$. 

\subsubsection{Monodromy representation}\label{sT.1.2} %{s1.2}
Recall that for a path-connected
locally contractible $X$ (and in more general situations, which would
not be of interest here) it is defined by the {\sfit monodromy
reprensentation\/} $\pi_1(X,x_0)\to\C^\times$, where 
$\C^\times=\C\sminus0$ is the
multiplicative group of $\C$. The monodromy representation assigns to
$\Gs\in\pi_1(X,x_0)$ a complex number $\Gz$ such that the parallel 
transport 
isomorphism along a loop which represents $\Gs$ is multiplication by $\Gz$.

Since $\C^\times$ is commutative, a
homomorphism $\pi_1(X,x_0)\to\C^\times$ factors through the 
abelianization
$\pi_1(X,x_0)\to H_1(X)$. Thus a local coefficient system with fiber
$\C$ is defined also by a homology version $\mu:H_1(X)\to\C^\times$ 
of the monodromy  representation, which can be considered also as a 
cohomology class belonging to $H^1(X;\C^\times)$. 

The local coefficient system defined by a monodromy representation
$\mu:H_1(X)\to\C^\times$ is denoted by $\C^{\mu}$. Sometimes instead of
$\mu$ we will write data which defines $\mu$, for example the images
under $\mu$ of generators of $H_1(X)$ selected in a special way.

\subsubsection{Twisted singular chains}\label{sT.1.3} %{s1.3}
Homology groups $H_n(X;\xi)$ of $X$ with coefficients in $\xi$ is a
classical invariant studied in algebraic topology. It is an immediate
generalization of $H_n(X;\C)$. Hence it is quite often ignored in
textbooks on homology theory, I recall the singular version of 
the definition.

Recall that a singular $p$-dimensional chain 
of $X$ with coefficients in $\C$ is a formal finite linear combination 
of singular simplices $f_i:T^p\to X$ with complex coefficients. 

A singular chain of $X$ with coefficients in $\xi$ is also a formal
finite linear combination of singular simplices, but each singular
simplex $f_i:T^p\to X$ appears 
in it with a coefficient taken from the fiber $\C_{f_i(c)}$ of $\xi$ 
over $f_i(c)$,
where $c$ is the baricenter of $T^p$. Of course, all
the fibers of $\xi$ are isomorphic to $\C$. So, a chain with
coefficients in $\xi$ can be identified with a chain with coefficients
in $\C$, provided the isomorphisms $\C_{f_i(c)}\to\C$ are selected. 
But they are not. 

All singular $p$-chains of $X$ with coefficients in $\xi$ form a 
complex vector space $C_p(X;\xi)$. 

The boundary of such a chain is defined by the usual formula, but one 
needs to bring the coefficient
from the fiber over  $f_i(c)$ to the fibers over $f_i(c_i)$, where $c_i$
is the baricenter of the $i$th face of $T^p$. 
For this, one may use translation along the composition with $f_i$ of 
any path connecting $c$ to $c_i$ in $T^p$: since $T^p$ is simply
connected and the connection of $\xi$ is flat, the result does not
depend on the path. 

These chains and boundary operators form a complex. Its homology is
called {\sfit homology with coefficients in\/} $\xi$ and denoted by 
$H_p(X;\xi)$. 

Homology with coefficients in the local coefficient system corresponding 
to the trivial monodromy representation $1:H_1(X)\to\C^\times$ coincides 
with homology with coefficients in $\C$.

\subsubsection{Twisted cellular chains}\label{sT.1.4} %{s1.4}
It is possible to calculate the homology with coefficients in a 
local coefficient system using cellular decomposition. Namely, a
$p$-dimensional cellular chain of a cw-complex $X$ with coefficients in
a local coefficient system $\xi$ is a formal finite linear combination
of $p$-dimensional cells in which a coefficient at a cell belongs to 
the fiber over a point of the cell. It does not matter which point is
this, because fibers over different points in a cell are identified via
parallel transport along paths in the cell: any two points in a cell 
can be connected in the cell by a path unique up to homotopy.

In order to describe the boundary operator, let me define the {\sfit
incidence number\/} $(z\Gs_x:\tau)_y\in\C_y$ where $\Gs$ is a $p$-cell, 
$\tau$ is a $(p-1)$-cell, $z\in\C_x$, $x\in\Gs$, $y\in\tau$.
The boundary operator is then defined by the incidence numbers:
$$\p(z\Gs)=\sum_\tau(z\Gs_x:\tau)_y\tau.$$

Let  $f:D^p\to X$ be a characteristic map for $\Gs$. Assume that a point
$y$ in $(p-1)$-cell $\tau$ is a regular value for $f$. This means that  
$y$ has a neighborhood $U$ in $\tau$ such that $f^{-1}(U)\subset
S^{p-1}\subset D^p$ is the union of finitely many balls mapped by $f$
homeomorphically onto $U$. Connect
$f^{-1}(x)\in D^p$ with all the points of $f^{-1}(y)$ by straight paths.
Compositions of these paths with $f$ are paths $s_1$,\dots $s_N$ 
connecting $x$ with $y$. Then put 
$$
(z\Gs:\tau)_y=\sum_{i=1}^N\Ge_iT_{s_i}(z)
$$
where $T_{s_i}$ is a parallel transport operator and 
$\Ge_i=+1$ or $-1$ according to whether $f$ preserves or reverses
the orientation on the $i$th ball out of $N$ balls constituting
$f^{-1}(U)$.

\subsection{Twisted acyclicity}\label{sT.2} %{s1-2}
\vspace{6pt}

\subsubsection{Acyclicity of circle}\label{sT.2.1} %{s1.5} 
According to one
of the most fundamental properties of homology, the dimension of 
 $H_0(X;\C)$ 
is equal to the number of path-connected components of $X$. In
particular, $H_0(X;\C)$ does not vanish, unless $X$ is empty.

This is not the case for twisted homology. A crucial example is the 
circle $S^1$.  Let $\mu:H_1(S^1)\to\C^\times$ maps the generator
$1\in\Z=H_1(S^1)$ to $\Gz\in\C^\times$. 

\begin{ATh}{\bfit Twisted acyclicity of circle.}\label{1.A}
$H_*(S^1;\C^\mu)=0$, iff \ $\Gz\ne 1$.
\end{ATh}

\begin{proof}The simplest cw-decomposition of $S^1$
consists of two cells, one-dimensional $\Gs_1$ and zero-dimensional
$\Gs_0$. One can easily see that $\p\Gs_1=(\Gz-1)\Gs_0$. Hence 
$\p:C_1(S^1;\C^\mu)\to
C_0(S^1;\C^\mu)$ is an isomorphism, iff $\Gz\ne0$. 
\end{proof}

\subsubsection{Vanishing of twisted homology}\label{sT.2.2} %{s1.5.5}
\begin{Acor}\label{1.B}
Let $X$ be a path connected space and 
$\mu:H_1(S^1\times X)\to\C^\times$ be a homomorphism. 
Denote by $\Gz$ the image under $\mu$ of the homology class 
realized by a fiber $S^1\times \text{point}$. 
Then $H_*(S^1\times X;\C^\mu)=0$, if \
 $\Gz\ne0$. 
\end{Acor}

\begin{proof} Since $H_1(S^1\times X)=H_1(S^1)\times H_1(X)$, the
homomorphism $\mu$ can be presented as product of homomorphisms
$\mu_1:H_1(S^1)\to\C^\times$ and $\mu_2:H_1(X)\to\C^\times$ which 
can be obtained
as compositions of $\mu$ with the inclusion homomorphisms. Thus
$\C^\mu=\C^{\mu_1}\otimes\C^{\mu_2}$, and we can apply K\"unneth
formula 
$$H_n(S^1\times X;\C^\mu)=\sum_{p=0}^n H_p(S^1;\C^{\mu_1})\otimes
H_{n-p}(X;\C^{\mu_2})$$
and refer to Theorem \ref{1.A}.
\end{proof}

\begin{Acor}\label{1.C}
Let $B$ be a path connected space, $p:X\to B$ a locally trivial
fibration with fiber $S^1$. Let $\mu:H_1(X)\to\C^\times$ be
a homomorphism. Denote by $\Gz$ the image under $\mu$ of homology class 
realized by a fiber of $p$. Then $H_*(X;\C^\mu)=0$, if \
 $\Gz\ne0$. 
\end{Acor}

\begin{proof} It follows from Theorem \ref{1.A} via the spectral sequence of
fibration $p$.\break
\end{proof}

\subsection{Estimates of twisted homology}\label{sT.3} %{s1-3}
\vspace{6pt}

\subsubsection{Equalities underlying the Morse inequalities}\label{sT.3.1}
%{s1.7}
\begin{Alem}\label{EqUnderlMI}
For a complex $C:\dots\to C_i\overset{\p_i}\to
C_{i-1}\to$ of finite dimensional vector spaces
over a field $F$
\begin{multline}\label{dimH}
 \sum_{s=r}^{2n+r}(-1)^{s-r}\dim_FH_s(C)=\\
 \sum_{s=r}^{2n+r}(-1)^{s-r}\dim_FC_s-\rnk\p_{r-1}-\rnk\p_{2n+r}.
\end{multline}  
\end{Alem}

\begin{proof} First, prove inequality \eqref{dimH} for $n=0$.
Since $H_s(C)=\Ker\p_s/\Im\p_{s+1}$, we have
$\dim_FH_s(C)=\dim\Ker\p_s-\dim_F\Im\p_{s+1}$. Further,
$\dim_F\Im\p_{s+1}=\rnk\p_{s+1}$, and
$\dim_F\Ker\p_s=\dim_FC_s-\rnk\p_s$. It follows
\begin{equation}\label{dimHs}
 \dim_FH_s(C)=\dim_FC_s-\rnk\p_s-\rnk\p{s+1}
\end{equation}
This is a special case of \eqref{dimH} with $n=0$, $r=s$.

The general case follows from it: make alternating summation of
\eqref{dimHs} for $s=r,\dots,2n+s$.
\end{proof}

\subsubsection{Algebraic Morse type inequalities}\label{sT.3.2} %{s1.7.1}
\begin{Alem}\label{AlgLem}
Let $P$ and $Q$ be fields, $R$ be a subring of $Q$ and let $h:R\to P$ be 
a ring homomorphism. Let $C: \dots\to C_p\to C_{p-1}\to\dots\to C_1\to C_0$
be a complex of free finitely generated $R$-modules. Then for any 
$n$ and $r$ 
$$\sum_{s=r}^{2n+r}(-1)^{s-r}\dim_QH_s(C\otimes_RQ)\le
 \sum_{s=r}^{2n+r}(-1)^{s-r}\dim_PH_s(C\otimes_hP)
$$
\end{Alem}

 Thus, the greater ranks of differentials, the smaller
$$\sum_{s=r}^{2n+r}(-1)^{s-r}\dim_FH_s(C).$$

\begin{proof} 
Choose free bases in modules $C_i$. Let $M_i$ be the matrix representing
 $\p_i:C_i\to C_{i-1}$ in these bases. The same matrix represents the 
differential $\p^Q_i$ of $C\otimes_RQ$.  The matrix obtained from $M_i$ by
replacement the entries with their images under $h$ represents the
differential $\p^P_i$ of $C\otimes_hP$. 
The minors of the latter matrix  are the images of the former one
under $h$. Consequently,  the $\rnk\p^Q_i\ge\rnk\p^P_i$. 

By Lemma \ref{EqUnderlMI}
\begin{multline}\label{eqQ}
 \sum_{s=r}^{2n+r}(-1)^{s-r}\dim_QH_s(C\otimes_RQ)= \\
  \sum_{s=r}^{2n+r}(-1)^{s-r}\dim_QC_s\otimes_RQ-\rnk\p^Q_{r-1}-\rnk\p^Q_{r+2n}
\end{multline}
and
 \begin{multline}\label{eqP}
 \sum_{s=r}^{2n+r}(-1)^{s-r}\dim_PH_s(C\otimes_hP)= \\
  \sum_{s=r}^{2n+r}(-1)^{s-r}\dim_PC_s\otimes_hP-\rnk\p^P_{r-1}-\rnk\p^P_{r+2n}
\end{multline} 

Compare the right hand sides of these equalities. 
The dimensions $\dim_PC_s\otimes_hP$, $\dim_QC_s\otimes_RQ$ are equal to
the rank of free $R$-module $C_s$.
Since, as it was shown above,  $\rnk\p^Q_i\ge\rnk\p^P_i$, the right hand
side of \eqref{eqP} is smaller than the right hands side of \eqref{eqP}.
\end{proof}

Probably, the simplest application of Lemma \ref{AlgLem} gives
well-known upper estimation of the Betti numbers with rational
coefficients by the Betti numbers with coefficients in a finite field. 
It follows from the universal coefficients formula. 

\subsubsection{Application to twisted homology}\label{sT.3.3} %{s1.7.2}
\begin{ATh}\label{EstTwHom}
Let $X$ be a finite cw-complex, and $\mu:H_1(X)\to\C^\times$ be a
homomorphism. If $\Im\mu\subset\C^\times$ generates a subring $R$ of $\C$ and
there is a ring homomorphism $h:R\to Q$, where $Q$ is a field, such that
$h\mu(H_1(X))=1$, then we can apply Lemma \ref{AlgLem} and get an upper
estimation for dimensions of twisted homology groups in terms of dimensions of
non-twisted ones.
\begin{equation}\label{twHomEst}
 \sum_{s=r}^{2n+r}(-1)^{s-r}\dim_QH_s(X;\C^\mu)\le
 \sum_{s=r}^{2n+r}(-1)^{s-r}\dim_PH_s(X;P) 
\end{equation}
\end{ATh} 

Here are several situations in which the assumptions of this theorem 
are fulfilled.

\subsubsection{Estimates by untwisted $\Z/p\Z$ Betti numbers}\label{sT.3.4} 
Let $H_1(X)$ be generated by $g$ and 
$\zeta=\mu(g)$ be an algebraic number.  Assume that $p$ is 
the minimal integer polynomial with relatively prime coefficients
which annihilates $\zeta$. Assume also that  $g(1)$ is divisible 
by a prime number $p$.
Then for $R$ we can take
$\Q[\zeta]\subset\C$, for $P$ the field $\Z/p\Z$, and for $h$ 
the ring homomorphism 
$\Q[\zeta]\to\Z/p\Z$ mapping $\zeta\mapsto1$.  

Here is a more general situation: Let $H_1(X)$ be generated by 
$g_1$,\dots $g_k$, and 
$\zeta_i=\mu(g_i)$ be an algebraic number for each $i$.  Assume that $p_i$ is 
the minimal integer polynomial with relatively prime coefficients
which annihilates $\zeta_i$. Assume also that the greatest common
divisor of $g_1(1)$,\dots, $g_k(1)$ is divisible by a prime number $p$.
Then for $R$ we can take
 $\Q[\zeta_1,\dots,\zeta_k]\subset\C$, for $P$ the field $\Z/p\Z$, and for $h$ 
the ring homomorphism 
$\Q[\zeta_1,\dots,\zeta_k]\to\Z/p\Z$ mapping $\zeta_i\mapsto1$ for all
$i$.   

\subsubsection{Estimates by rational Betti numbers}\label{sT.3.5} Let $H_1(X)$ 
be generated by $g$ and $\zeta=\mu(g)$ be transcendent.
Then for $R$ we can take the ring $\Z[\zeta,\Gz^{-1}]$, 
for $Q$ the field $\Q(\zeta)$, 
for $P$ the field $\Q$, and for $h$ the ring homomorphism $\Z[\zeta]\to\Q$
which maps $\zeta$ to 1. 

\subsubsection{The most general estimates}\label{sT.3.6} 
Let $H_1(X)$ be generated by $g_1$,\dots $g_k$ and $\Gz_i=\mu(g_i)$.  
Laurent polynomials with integer coefficients annihilated by  
$\Gz_1,\dots,\Gz_m$ form an ideal in the ring 
$\Z[t_1,t_1^{-1}\dots,t_m,t_m^{-1}]$. Let $p_1,\dots,p_k$ be generators of
this ideal. 
Let $d$ be the greatest common divisor of the integers 
$p_1(1,\dots,1)$, \dots, $p_k(1,\dots,1)$, if at least one of them is 
not 0. Otherwise, let $d=0$ 

In other words, consider the specialization homomorphism
$$S:\Z[t_1,t_1^{-1}\dots,t_m,t_m^{-1}]\to \C: t_i\mapsto\Gz_i.$$
Let $K$ be the kernel of $S$, and let $d$ be the generator of the ideal 
which is the image of $K$ under the homomorphism 
$$\Z[t_1,t_1^{-1}\dots,t_m,t_m^{-1}]\to\Z : t_i\mapsto 1.$$

Then for $R$ we can take the ring
$\Z[\Gz_1,\Gz_1^{-1},\dots,\Gz_k,\Gz_k^{-1}]$. For $Q$ we can take
the quotient field of $R$, but since both $Q$ and its quotient field
are contained in $\C$, let us take $Q=\C$.

If $d>1$, then we can take for $P$ the field $\Z/p\Z$ with any prime $p$
which divides $d$.  
If $d=0$, then let $P=\Q$. The case $d=1$ is the most misfortunate:
then our technique does not give any non-trivial estimate.
For $d>1$ or $d=0$ we have the inequality \eqref{twHomEst}.

\subsection{Twisted duality}\label{sT.4} %{s1-4}
\vspace{6pt}

\subsubsection{Cochains and cohomology}\label{sT.4.1} %{s1.8}
Cochain groups $C^p(X;\xi)$ (which are vector spaces over
$\C$) and cohomology $H^p(X;\xi)$ are defined similarly:   
$p$-cochain with coefficients in $\xi$ is a function assigning to a
singular simplex $f:T^p\to X$ an element of $\C_{f(c)}$, the 
fiber of $\xi$ over $f(c)$. 

This can be interpreted as the chain complex of the local coefficient
system $\Hom(\C,\xi)$ whose fiber over $x\in X$ is $\Hom_\C(\C,\C_x)$. 
More generally, for any local coefficient systems $\xi$ and $\eta$ on $X$ 
with fiber $\C$ there is a local coefficient system $\Hom(\xi,\eta)$
constructed fiber-wise with the parallel transport defined naturally in terms of
the parallel transports of $\xi$ and $\eta$. If the monodromy
representations of $\xi$ and $\eta$ are $\mu$ and $\nu$, respectively,
then the monodromy representation of $\Hom(\xi,\eta)$ is
$\mu^{-1}\nu:H_1(X)\to\C^\times:x\mapsto\mu^{-1}(x)\nu(x)$.

Similarly,  for any local coefficient systems $\xi$ and $\eta$ on $X$ 
with fiber $\C$  there is a local coefficient system $\xi\otimes \eta$.
If $\mu,\nu: H^1(X)\to\C^\times$ are homomorphisms, then 
$\C^\mu\otimes\C^\nu$ is the local coefficient system $\C^{\mu\nu}$ 
corresponding to the homomorphism-product 
$\mu\nu:H^1(X)\to\C^\times:x\mapsto \mu(x)\nu(x)$.
                                              
If $\nu=\mu^{-1}$ (that is $\mu(x)\nu(x)=1$ for any $x\in H^1(X)$), 
then $\C^\mu\otimes\C^\nu$ is the non-twisted coefficient system 
with fiber $\C$. 

In contradistinction to non-twisted case, there is no way to calculate 
$H_n(X;\xi\otimes\eta)$ in terms of $H_*(X;\xi)$ and $H_*(X;\eta)$.
Indeed, both $H_*(S^1;\C^\mu)$ and $H_*(S^1;\C^{\mu^{-1}})$ vanish, 
unless $\mu:H_1(S^1)\to\C^\times$ is trivial, but
$H_0(S^1;\C^\mu\otimes\C^{\mu^{-1}})=H_0(S^1;\C)=\C$. 

\subsubsection{Multiplications}\label{sT.4.2} %{s1.5}
Usual definitions of various cohomological and homological
multiplications are easily generalized to twisted homology. For this one
needs a bilinear pairing of the coefficient systems. (Recall that in the
case of non-twisted coefficient system a pairing of coefficient groups
also is needed.) For local coefficient systems $\xi$, $\eta$ and $\zeta$
with fiber $\C$ on $X$, a pairing $\xi\oplus\eta\to\zeta$ is a fiber-wise map
which is bilinear over each point of $X$. Given such a pairing, there
are pairings 
$$\smallsmile:H^p(X;\xi)\times H^q(X;\eta)\to H^{p+q}(X;\zeta),$$
 $$\smallfrown:H^{p+q}(X;\xi)\times H^q(X;\eta)\to H^{p}(X;\zeta),$$ 
 etc.

A pairing $\xi\oplus\eta\to\zeta$ of local coefficients systems can be 
factored through the universal pairing $\xi\oplus\eta\to\xi\otimes\eta$.

Since $\C^\mu\otimes\C^{\mu^{-1}}$ is a non-twisted coefficient system
with fiber $\C$, this gives rise to a non-singular pairing
$$
 C_p(X;\C^{\mu^{-1}})\otimes C^p(X;\C^\mu)\to \C 
$$
which induces a non-singular pairing 
$$
\smallfrown: H_p(X;\C^{\mu^{-1}})\otimes H^p(X;\C^\mu)\to \C
$$
Thus, the vector spaces $H_p(X;\C^{\mu^{-1}})$ and
$H^p(X;\C^\mu)$ are dual.

\subsubsection{Poincar\'e duality}\label{sT.4.3} %{s1.9}
Let $X$ be an oriented connected compact manifold of dimension $n$.
Then $H_n(X,\p X)$ is isomorphic to $\Z$ and the orientation is a 
choice of the isomorphism, or, equivalently, the choice of a generator
of $H_n(X,\p X)$. We denote the generator by $[X]$.

Let $\mu:H_1(X)\to\C^\times$ be a homomorphism. There are the
Poincar\'e-Lefschetz duality isomorphisms
$$[X]\smallfrown :H^p(X;\C^\mu)\to H_{n-p}(X,\p X;\C^{\mu}),
$$
$$ [X]\smallfrown :H^p(X,\p X;\C^\mu)\to H_{n-p}(X;\C^{\mu}) 
$$

Similarly to the case of non-twisted coefficients, there are 
non-singular pairings:
the cup-product pairing 
$$
\smallsmile:H^p(X;\C^\mu)\times H^{n-p}(X,\p X;\C^{\mu^{-1}})\to
H^n(X;\C)=\C
$$
and intersection pairing
\begin{equation}\label{bilin-ip}
\circ:H_p(X;\C^\mu)\times H_{n-p}(X,\p X;\C^{\mu^{-1}})\to \C
\end{equation}
However, the local coefficient systems of the homology or cohomology
groups involved in a pairing are different,
unless $\Im\mu\subset\{\pm1\}$. 

\subsubsection{Conjugate local coefficient systems}\label{sT.4.4} %{s1.10}
Recall that for vector spaces $V$ and
$W$ over $\C$ a map $f:V\to W$ is called semi-linear if $f(a+b)=f(a)+f(b)$
for any $a,b\in V$ and $f(za)=\overline zf(a)$ for $z\in\C$ and $a\in
V$. This notion extends obviously to fiber-wise maps of complex vector
bundles. If $\xi$ and $\eta$ local coefficient systems of the type that
we consider, then fiber-wise semi-linear bijection $\xi\to\eta$ commuting 
with all the transport maps is called a {\sfit semi-linear
equivalence\/} between $\xi$ and $\eta$.

For any local coefficient system $\xi$ with fiber $\C$ on $X$ there
exists a unique local coefficient system on $X$ which is
semi-linearly equivalent to  $\xi$. It is denoted by $\overline\xi$ and
called {\sfit conjugate\/} to $\xi$.
If $\xi=\C^\mu$, then $\overline\xi$ is $\C^{\overline\mu}$, where
$\overline\mu(x)=\overline{\mu(x)}$ for any $x\in H_1(X)$.

\subsubsection{Unitary local coefficient systems}\label{sT.4.5} %{s1.11}
A homomorphism $\mu:H_1(X)\to\C^\times$ is called {\sfit unitary\/} if
$\Im\mu\subset S^1=U(1)=\{z\in\C\mid |z|=1\}$. In $S^1$ the inversion
$z\mapsto z^{-1}$ coincides with the complex conjugation: if $|z|=1$, then
$z^{-1}=\overline z$. Therefore if $\mu:H_1(X)\to\C^\times$ is unitary, 
then $\overline{\C^\mu}=\C^{\mu^{-1}}$ and there exists a 
{\sfit semi-linear\/ } equivalence $\C^\mu\to\C^{\mu^{-1}}$. 

This semi-linear equivalence induces semi-linear equivalence
$$H_{k}(X;\C^\mu)\to H_k(X;\C^{\mu^{-1}})$$ 
and similar semi-linear equivalences in cohomology and relative 
homology and cohomology.

Combining a semi-linear isomorphism 
$$H_{n-p}(X,\p X;\C^\mu)\to H_{n-p}(X,\p X;\C^{\mu^{-1}})$$ 
of this kind with the intersection
pairing \eqref{bilin-ip} we get a {\sfit sesqui-linear \/} pairing
\begin{equation}\label{ssqlin-ip}
 \circ:H_p(X;\C^\mu)\times H_{n-p}(X,\p X;\C^{\mu})\to \C 
\end{equation}
(Sesqui-linear means that it is linear on the first variable, and
semi-linear on the second one.) This pairing is non-singular, because
the bilinear pairing  \eqref{bilin-ip} is non-singular, and  \eqref{ssqlin-ip}
differs from it by a semi-linear equivalence on the second variable.

\subsubsection{Intersection forms}\label{sT.4.6} %{s1.12}
Let $X$ be an oriented connected compact smooth manifold of even 
dimension $n=2k$ and $\mu:H_1(X)\to\C^\times$ be a unitary homomorphism.
Combining the relativisation homomorphism 
$$
H_{n-p}(X;\C^{\mu})\to H_{n-p}(X,\p X;\C^\mu)
$$
with the pairing \eqref{ssqlin-ip} for $p=k$  define sesqui-linear form
\begin{equation}\label{ssqlin-if}
\circ:H_k(X;\C^\mu)\times H_k(X;\C^\mu)\to\C
\end{equation}
It is called the {\sfit intersection form\/} of $X$. 

If $k$ is even, this form is {\sfit Hermitian\/}, that is $\Ga\circ\Gb=\overline{\Gb\circ\Ga}$.
If $k$ is odd, it is {\sfit skew-Hermitian\/}, that is
$\Ga\circ\Gb=-\overline{\Gb\circ\Ga}$. 

The difference between Hermitian and skew-Hermitian forms is not as deep
as the difference between symmetric and skew-symmetric bilinear forms.
Multiplication by $i=\sqrt{-1}$ turns a skew-Hermitian form into a
Hermitian one, and the original form can be recovered. In order to
recover, just multiply the Hermitian form by $-i$. 

The intersection form \eqref{ssqlin-if} may be singular. Its radical,
that is the orthogonal complement of the whole $H_k(X;\C^\mu)$, is the
kernel of the relativisation homomorphism  
$H_k(X;\C^{\mu})\to H_k(X,\p X;\C^\mu)$.  
It can be described also as the image of the inclusion homomorphism 
$$H_k(\p X;\C^{\mu\inc_*})\to H_k(X;\C^\mu),$$ where $\inc_*$ is the
inclusion homomorphism $H_1(\p X)\to H_1(X)$.

\subsubsection{Twisted signatures and nullities}\label{sT.4.7} %{s1.13}
As well-known for any Hermitian form on a finite-dimensional space
$V$ there exists an orthogonal basis in which the form is represented by 
a diagonal matrix. The diagonal entries of the matrix are real.
The number of zero diagonal entries is called the {\sfit nullity\/}, 
and the difference between the number of positive and negative
entries is called the {\sfit signature\/} of the form.
These numbers do not depend on the basis.

For a skew-Hermitian form by nullity and signature one means the nullity
and signature of the Hermitian form obtained by multiplication of the
skew-Hermitian form by $i$.

For a compact oriented $2k$-manifold $X$ and a homomorphism $\mu:H_1(X)\to\C$
the signature and nullity of the intersection form
$$\circ:H_k(X;\C^\mu)\times H_k(X;\C^\mu)\to\C  $$
are denoted by $\Gs_\mu(X)$ and $n_\mu(X)$, respectively, and called the
{\sfit twisted\/} signature and nullity of $X$.

The classical theorems about the signatures of the symmetric intersection 
forms of oriented compact $4k$-manifolds are easily generalized to
twisted signatures:

\begin{ATh}{\bfit Additivity of Signature.}\label{AddOfSign}
Let $X$ be an oriented compact manifold of even dimension.
If $A$ and $B$ are its compact submanifolds of the same dimension
such that $A\cup B=X$, $\Int A\cap\Int B=\varnothing$ and $\p(A\cap
B)=\varnothing$, then for any $\mu:H_1(X)\to\C^\times$ 
$$\Gs_\mu(X)=\Gs_{\mu\inc_*}(A)+\Gs_{\mu\inc_*}(B)$$
where $\inc$ denotes an appropriate inclusion.
\end{ATh}

\begin{ATh}{\bfit Signature of Boundary.}\label{VanSign}
Let $X$ be an oriented compact manifold of odd dimension. Then 
$\Gs_{\mu\inc_*}(\p X)=0$
for any
$\mu:H_1(X)\to\C^\times$.
\end{ATh}

\end{document}